\documentclass[11pt]{article}
\usepackage{amsfonts,amsthm, amsmath}
\usepackage{epsfig}
\usepackage{makeidx}
\usepackage{graphicx}
\usepackage{graphicx,epstopdf}
\usepackage{amsmath}
\usepackage{comment}
\usepackage{enumerate}
\DeclareGraphicsExtensions{.ps,.png,.jpg}

\makeindex
\usepackage{color}

\def\E{{\mathbb E}}

\newcommand{\ncom}{\newcommand}
\ncom{\ul}{\underline}
\ncom{\beq}{\begin{equation}}
\ncom{\eeq}{\end{equation}}
\ncom{\bea}{\begin{eqnarray*}}
\ncom{\eea}{\end{eqnarray*}}
\ncom{\beqa}{\begin{eqnarray}}
\ncom{\eeqa}{\end{eqnarray}}
\ncom{\nno}{\nonumber}
\ncom{\non}{\nonumber}
\ncom{\ds}{\displaystyle}
\ncom{\half}{\frac{1}{2}}
\ncom{\mbx}{\makebox{.25cm}}
\ncom{\hs}{\mbox{\hspace{.25cm}}}
\ncom{\rar}{\rightarrow}
\ncom{\Rar}{\Rightarrow}
\ncom{\noin}{\noindent}
\ncom{\bc}{\begin{center}}
\ncom{\ec}{\end{center}}
\ncom{\sz}{\scriptsize}
\ncom{\rf}{\ref}
\ncom{\s}{\sqrt{2}}
\ncom{\sgm}{\sigma}
\ncom{\Sgm}{\Sigma}
\ncom{\psgm}{\sigma^{\prime}}
\ncom{\dt}{\delta}
\ncom{\Dt}{\Delta}
\ncom{\lmd}{\lambda}
\ncom{\Lmd}{\Lambda}
\ncom{\Th}{\Theta}
\ncom{\e}{\eta}
\ncom{\eps}{\epsilon}
\ncom{\pcc}{\stackrel{P}{>}}
\ncom{\lp}{\stackrel{L_{p}}{>}}
\ncom{\dist}{{\rm\,dist}}
\ncom{\sspan}{{\rm\,span}}
\ncom{\re}{{\rm Re\,}}
\ncom{\im}{{\rm Im\,}}
\ncom{\sgn}{{\rm sgn\,}}
\ncom{\ba}{\begin{array}}
\ncom{\ea}{\end{array}}
\ncom{\hone}{\mbox{\hspace{1em}}}
\ncom{\htwo}{\mbox{\hspace{2em}}}
\ncom{\hthree}{\mbox{\hspace{3em}}}
\ncom{\hfour}{\mbox{\hspace{4em}}}
\ncom{\vone}{\vskip 2ex}
\ncom{\vtwo}{\vskip 4ex}
\ncom{\vonee}{\vskip 1.5ex}
\ncom{\vthree}{\vskip 6ex}
\ncom{\vfour}{\vspace*{8ex}}
\ncom{\norm}{\|\;\;\|}
\ncom{\integ}[4]{\int_{#1}^{#2}\,{#3}\,d{#4}}
\ncom{\vspan}[1]{{{\rm\,span}\{ #1 \}}}
\ncom{\dm}[1]{ {\displaystyle{#1} } }
\ncom{\ri}[1]{{#1} \index{#1}}

\newtheorem{remark}{\bf Remark}[section]
\newtheorem{proposition}{Proposition}[section]

\newtheorem{corollary}{Corollary}[section]

\newtheoremstyle
    {remarkstyle}
    {}
    {11pt}
    {}
    {}
    {\bfseries}
    {:}
    {     }
    {\thmname{#1} \thmnumber{#2} }

\theoremstyle{remarkstyle}

\begin{document}
%\input{iitbcover3.tex}
%\pagenumbering{roman}
%\input{cover1.txt}
%\input{cert.tex}
%\newpage
%\input{ack.tex}
%\tableofcontents
%\newpage
%\pagenumbering{arabic}
%\newpage
%\tableofcontains

\newpage

\begin{center}
{\Large \bf Mixtures of Tempered Stable Subordinators}
\end{center}
\vone
\begin{center}
{Neha Gupta}$^{\textrm{a}}$, {Arun Kumar}$^{\textrm{a}}$, {Nikolai Leonenko}$^{\textrm{b}}$

\footnotesize{
		$$\begin{tabular}{l}
		$^{\textrm{a}}$ \emph{Department of Mathematics, Indian Institute of Technology Ropar, Rupnagar, Punjab - 140001, India}\\

$^{b}$Cardiff School of Mathematics, Cardiff University, Senghennydd Road,
Cardiff, CF24 4AG, UK

\end{tabular}$$}
\end{center}
\vtwo

\begin{center}
\noindent{\bf Abstract}
\end{center}
In this article, we introduce mixtures of tempered stable subordinators (MTSS). These mixtures define a class of subordinators which generalize tempered stable subordinators (TSS). The main properties like probability density function (pdf), L\'evy density, moments, governing Fokker-Planck-Kolmogorov (FPK) type equations, asymptotic form of potential density and asymptotic form of the renewal function are discussed. We generalize these results to $n$-th order mixtures of TSS. We also discuss the time-changed Poisson process and Brownian motion where the introduced mixture process and its inverse are used as time-changes. Our results extend and complement the results available in literature on TSS and time-changed Poisson processes in several directions.\\
\vone \noindent{\it Key words:} Tempered stable subordinator, mixture, L\'evy density, Fokker-Planck-Kolmogorov equations.
\vtwo
%\vtwo
\setcounter{equation}{0}

\section{Introduction}
In recent years the subordinated stochastic processes have found many interesting real life applications (see e.g. Mandelbrot et al. 1997;  Barndorff-Nielsen, 1998; Heyde, 1999; Gabaix et al., 2003; Stanislavsky et al., 2008; Meerschaert et al. 2011, 2013; Leonenko et al. 2014) and references therein. In general, a subordinated process is defined by taking superposition of two independent  stochastic processes. In subordinated process the time of a process called parent process (or outer process) is replaced by another independent stochastic process called as inner process or subordinator. A subordinator is a non-decresing L\'evy process (see Applebaum, 2009). Note that subordinated processes are convenient way to develop a stochastic model where it is required to keep some properties of the parent process and at the same time some characteristics are need to be altered. Some well known subordinators include gamma process, Poisson process, one-sided stable process with index $\alpha\in(0,1)$ or $\alpha$-stable subordinator, tempered stable subordinators, geometric stable subordinators, iterated geometric stable subordinators and Bessel subordinators (see e.g. Bogdan et al. 1980).

In this article, we introduce a class of subordinators which generalize the class of tempered stable subordinators and $\alpha$-stable subordinators. This class of subordinators can be used as a time-change to define another subordinated process instead of tempered stable subordinator or $\alpha$-stable subordinator. We have discussed the main properties like probability density function (pdf), L\'evy density, moments, governing Fokker-Planck-Kolmogorov (FPK) type equations, asymptotic form of potential density and asymptotic form of the renewal function for the introduced subordinator. Further, we discuss time-changed Poisson process and Brownian motion where introduced subordinator and its inverse are used as time-changes. Our results extend and complement the results available in Orsingher and Polito (2012) and Meerschaert et al. (2011). 

The rest of the paper is organized as follows. In Section  2, we introduce $\alpha$-stable and tempered stable subordinators (TSS) and also indicate their main characteristics. In Section 3, the mixtures of TSS are defined and their main properties are discussed. Section 4, deals with $n$-th order mixtures of TSS. In the last section, we introduce time-changed Poisson process and Brownian motion by considering the mixtures of tempered stable subordinator and its inverse as time-change.

\setcounter{equation}{0}
\section{Tempered Stable Subordinators}
In this section, we recall the definitions of the $\alpha$-stable subordinator as well as the tempered stable subordinator. Moreover, we present main properties of these processes. The class of stable distributions is denoted by $S(\alpha,\beta,\mu,\sigma)$, with parameter $\alpha \in(0,2]$ is the stability index, $\beta \in[-1,1]$ is the skewness parameter, $\mu\in\mathbb{R}$ is the location parameter and $\sigma>0$ is the shape parameter. The stable calss probability density
functions do not possess closed form except for three cases (Gaussian, Cauchy, and
L\'evy). Generally stable distributions are represented in terms of their characteristic functions or Laplace transforms. Stable distributions are infinitely divisible and hence generate a class of continuous time L\'evy processes. The one-sided stable L\'evy process $S_{\alpha}(t)$ with Laplace transform (see Samorodnitsky and Taqqu, 1994)
\begin{equation}
\mathbb{E}(e^{-s S_{\alpha}(t)}) = e^{-t s^{\alpha}},\;s>0,\;\alpha\in(0,1),
\end{equation}
is called the $\alpha$-stable subordinator. The $\alpha$-stable subordinator $S_{\alpha}(t)$ has stationary independent increments. The right tail of the $\alpha$-stable subordinator behaves (Samorodnitsky and Taqqu, 1994)
\begin{equation}\label{tail-stable}
\mathbb{P}(S_{\alpha}(t)>x) \sim \frac{t x^{-\alpha}}{\Gamma{(1-\alpha)}},\;\mbox{as}\;x\rightarrow\infty.
\end{equation}
Next, we introduce the tempered stable subordinator (TSS). The TSS $S_{\alpha, \lambda}(t)$  with tempering parameter $\lambda>0$ and stability index $\alpha \in(0,1)$, is the L\'evy process with Laplace transform (LT) (see Meerschaert et al., 2013)
\beq\label{tem-LT}
\mathbb{E}(e^{-s S_{\alpha,\lambda}(t)})=
e^{-t\big((s + \lambda)^{\alpha}-\lambda^{\alpha}\big)}. \eeq Note that TSS are obtained by exponential tempering in the distributions of $\alpha$-stable subordinators  (see e.g. Rosinski, 2007). The advantage of tempered stable subordinator over an $\alpha$-stable subordinator is that it has moments of all order and its density is also infinitely divisible. However in process of tempering it ceases to be self-similar.  The probability density function for $S_{\alpha, \lambda}(t)$ is given by 
\beq
\label{ts-density} f_{\alpha, \lambda}(x,t)= e^{-\lambda x+\lambda^{\alpha}t} f_{\alpha}(x,t),~~ \lambda>0, \;\alpha\in (0,1), 
\eeq
where 
$$f_{\alpha}(x,t) = \frac{1}{\pi}\sum_{k=1}^{\infty}(-1)^{k+1}\frac{\Gamma(\alpha k+1)}{k!}\frac{t^k}{x^{\alpha k+1}}\sin(\pi\alpha k),\;x>0,$$
 is the PDF of an $\alpha$-stable subordinator (see e.g. Uchaikin and Zolotarev $1999$). 
The L\'evy density  corresponding to a TSS is given by (see e.g. Cont and Tankov, 2004, p. 115)
\begin{equation}\label{tss-levyDensity} 
\pi_{S_{\alpha, \lambda}}(x) =  \frac{\alpha}{\Gamma(1-\alpha)}\frac{e^{-\lambda x}}{x^{\alpha+1}}, ~x>0.
\end{equation}
The sample paths of $\alpha$-stable subordinator and TSS are strictly increasing with jumps by an application of Theorem 21.3 of Sato (1999).
The tail probability of TSS has the following asymptotic behavior
\begin{align}\label{tail-TSS}
\mathbb{P}(S_{\alpha, \lambda}(t) > x) &\sim c_{\alpha,\lambda,t}\frac{e^{-\lambda x}}{x^{\alpha}},\;\mbox{as}\;x\rightarrow \infty,
\end{align}
where $c_{\alpha,\lambda,t} = \frac{t}{\alpha\pi}\Gamma(1+\alpha)\sin(\pi\alpha)e^{\lambda^{\alpha}t}.$
The first two moments and covariance of the TSS are given by
\begin{equation}\label{moments-tss}
\mathbb{E}(S_{\alpha, \lambda}(t)) = \alpha \lambda^{\alpha-1}t, \;\; \mathbb{E}(S_{\alpha, \lambda}(t))^2 = \alpha(1-\alpha) \lambda^{\alpha-2}t + (\alpha \lambda^{\alpha-1}t)^2, 
\end{equation}
$${\rm Cov}(S_{\alpha, \lambda}(t), S_{\alpha, \lambda}(s))= \alpha(1-\alpha)\lambda^{\alpha-2}\min(t, s), \; t,s\geq 0. $$

\section{Mixtures of TSS}
In this section, the mixtures of TSS are introduced and their main properties are discussed. 
\subsection{Definition}
Mixture of inverse stable subordinators have been considered in (see Aletti, et. al., $2018$). We define a mixture tempered stable subordinator (MTSS) denoted by $S_{\alpha_1,\lambda_1,\alpha_2,\lambda_2}(t),\; t\geq 0$, as a L\'evy process with Laplace transform
\begin{equation}\label{mtss}
\mathbb{E}\left(e^{-s S_{\alpha_1,\lambda_1,\alpha_2,\lambda_2}(t)}\right) = e^{-t\left(c_1\left((s+\lambda_1)^{\alpha_1}-\lambda_1^{\alpha_1}\right) + c_2\left((s+\lambda_2)^{\alpha_2}-\lambda_2^{\alpha_2}\right)\right)},\;s> 0,
\end{equation}
where $c_1+c_2 =1$ and $c_1,\; c_2\geq 0$. An alternative representation of MTSS can be given as a sum of two independent tempered stable subordinators $S_{\alpha_1, \lambda_1}(t)$ and $S_{\alpha_2, \lambda_2}(t)$ with time scaling and the condition $c_1+c_2 =1$, such that
\begin{equation}\label{alternative-rep}
S_{\alpha_1,\lambda_1,\alpha_2,\lambda_2}(t) = S_{\alpha_1,\lambda_1}(c_1 t) + S_{\alpha_2,\lambda_2}(c_2 t),\; c_1,\;c_2\geq 0.
\end{equation}
The representation \eqref{alternative-rep} directly follows from \eqref{mtss} and using the Laplace transforms of the tempered stable subordinators $S_{\alpha_1, \lambda_1}(t)$ and $S_{\alpha_2, \lambda_2}(t)$ and the fact that both processes in the LHS and RHS are L\'evy processes and hence the equivalence of their one-dimensional distributions lead to the equivalence of two processes. Further, the sample paths of MTSS are strictly increasing since sample paths of independent TSS used in \eqref{alternative-rep} are strictly increasing.

\subsection{Probability density function (pdf)}
We discuss the pdf $g_{\alpha_1, \lambda_1, \alpha_2, \lambda_2}(x,t)$ of introduced strictly incresing L\'evy process $S_{\alpha_1,\lambda_1,\alpha_2,\lambda_2}(t)$ . Here we use the technique of complex inversion of the Laplace Transform (LT) for finding the pdf of MTSS.

\begin{proposition} The pdf $g_{\alpha_1, \lambda_1, \alpha_2, \lambda_2}(x,t)$ of MTSS defined in \eqref{mtss} is given by the following integral representation,\\ 
if $\lambda_{1}\neq \lambda_{2}$,
\begin{align}\label{withcondpd}
g_{\alpha_1, \lambda_1, \alpha_2, \lambda_2}&(x,t) =\frac{1}{\pi}\int_{0}^{\infty}e^{-x\lambda_2}e^{-wx} e^{t(c_1\lambda_{1}^{\alpha_1}+c_2\lambda_2^{\alpha_2})}
\times e^{-t\left(c_1(\lambda_1-\lambda_2)^{\alpha_{1}}\sum_{k=0}^{\infty}{{{\alpha_{1}}\choose{k}}\frac{w^{k}}{(\lambda_{1}-\lambda_{2})^{k}}\cos{\pi k}}+c_2w^{\alpha_2}\cos{\pi \alpha_{2}}\right)} \nonumber\\
&\times\sin{\left(c_1 t(\lambda_1-\lambda_2)^{\alpha_{1}}\sum_{k=0}^{\infty}{{{\alpha_{1}}\choose{k}}\frac{w^{k}}{(\lambda_{1}-\lambda_{2})^{k}}\sin{(\pi k)}}+c_2 t w^{\alpha_2}\sin{(\pi \alpha_{2})}\right)}dw\nonumber\\
&+\frac{1}{\pi}\int_{0}^{\lambda_2-\lambda_1}e^{-x\lambda_1}e^{-wx} e^{t(c_1\lambda_1^{\alpha_1}+c_2\lambda_2^{\alpha_2})}\times e^{-t\left(c_1w^{\alpha_1} \cos(\pi \alpha_1)+c_2(\lambda_1-\lambda_2)^{\alpha_{2}}\sum_{k=0}^{\infty}{{{\alpha_{2}}\choose{k}}\frac{w^{k}}{(\lambda_{1}-\lambda_{2})^{k}}\cos{(\pi k)}}\right)}\nonumber\\
&\times \sin\left(c_1 t w^{\alpha_1}\sin{(\pi \alpha_1)}+c_1t (\lambda_1-\lambda_2)^{\alpha_{2}}\sum_{k=0}^{\infty}{{{\alpha_{2}}\choose{k}}\frac{w^{k}}{(\lambda_{1}-\lambda_{2})^{k}}\sin{(\pi k)}}\right)dw,
\end{align}
if $\lambda_{1},\lambda_{2}=\lambda$,
\begin{align}\label{comPDF}
g_{\alpha_1, \lambda_1, \alpha_2, \lambda_2}(x,t) & =\frac{1}{\pi}\int_{0}^{\infty}e^{-x\lambda}e^{-wx} e^{t(c_1\lambda^{\alpha_1}+c_2\lambda^{\alpha_2})}\times e^{-t\left(c_1w^{\alpha_{1}}\cos{(\pi \alpha_{1})}+c_{2}w^{\alpha_{2}}\cos{(\pi \alpha_{2})}\right)}\nonumber\\
&\times \sin{\left( t(c_1w^{\alpha_{1}}\sin{(\pi \alpha_{1})}+c_{2}w^{\alpha_{2}}\sin{(\pi \alpha_{2})})\right)}dw
\end{align}
where $c_{1}+c_{2}=1$ and $c_{1},c_{2}\geq0$.
\end{proposition}

\begin{proof}
Let $\mathcal{L}_x (f(x,t))$ be the LT of the function $f(x,t)$ with respect to the $x$ variable. Then for $g_{\alpha_1, \lambda_1, \alpha_2, \lambda_2}(x,t)$, we have
\begin{equation}\label{PDF}
\mathcal{L}_{x}\left(g_{\alpha_1, \lambda_1, \alpha_2, \lambda_2}(x,t)\right)= e^{-t\left(c_1\left((s+\lambda_1)^{\alpha_1}-\lambda_1^{\alpha_1}\right) + c_2\left((s+\lambda_2)^{\alpha_2}-\lambda_2^{\alpha_2}\right)\right)}={\overline{G}}(s,t) (say).
\end{equation}
The pdf $g_{\alpha_1, \lambda_1, \alpha_2, \lambda_2}(x,t)$ can be obtained by using the Laplace inverse formula, namely 
(see e.g. Schiff, 1999)
\begin{equation}\label{contour}
g_{\alpha_1, \lambda_1, \alpha_2, \lambda_2}(x,t) = \frac{1}{2\pi i}\int_{x_{0}-i\infty}^{x_{0}+i\infty}e^{sx}\overline{G}(s,t)ds. \\
\end{equation}
We assume here $\lambda_1, \lambda_2 \geq{0}$, for calculating integral in \eqref{contour}, consider a closed double-key-hole contour $C$: ABCDEFGHIJA (Fig. 1) with branch points at $P = (-\lambda_1, 0)$ and $Q = (-\lambda_2, 0)$. In the contour, AB and IJ are arcs of radius $R$ with center at $O=(0,0)$, BC, DE, FG and HI are line segment parallel to $x$-axis, CD, EF and GH are arcs of circles with radius $r$ and JA is the line segment from $x_{0}-iy$ to $x_{0}+iy$ (see Fig. 1). By residue theorem (Schiff, 1999),
\begin{equation}\label{resdue}
\frac{1}{2\pi i}\int_{C}e^{sx}\overline{G}(s,t)ds=\sum \mbox{Res}\ e^{sx} \overline{G}(s,t).
\end{equation}

\begin{figure}[ht!]
\centering{\includegraphics[scale=.4]{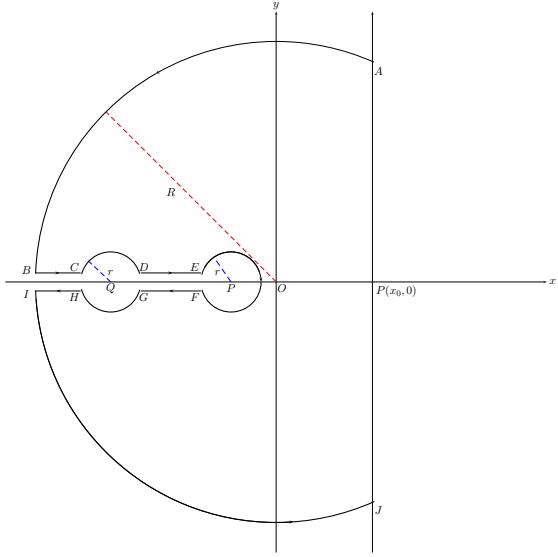}}
\caption{Contour ABCDEFGHIJA}
\end{figure}

\noindent The right hand side in \eqref{resdue} is $0$, since the function has no simple pole. On evaluation, we find that integral over AB, CD, EF, GH and IJ tend to zero as the radius $R$ goes to $\infty$ and $r$ goes to zero. So, we only have the remaining line integrals along BC, DE, FG and HI.
Along BC, we have $s= -\lambda_2+we^{i\pi}$, which implies $ds=-dw$ and 
%$$\int_{BC}e^{sx}F(s,t)ds=\int_{-R}^{-\lambda_2-r}e^{sx} e^{-t[c_{1}(s+\lambda_{1})^{\alpha_1}+c^{2}(s+\lambda_2)^{\alpha_2}]}ds\\$$
\begin{align}
\int_{BC}e^{sx}\overline{G}(s,t)ds = \int_{r}^{-\lambda_2+R}e^{-x\lambda_2}e^{-wx} e^{t(c_1\lambda^{\alpha_1}+c_2\lambda_2^{\alpha_2})}e^{-t[c_1(\lambda_1-\lambda_2+we^{i\pi})^{\alpha_1}+c_2(w^{\alpha_2}e^{i\pi \alpha_2})]}dw.
\end{align}
Similarly, along DE, we have, $s=-\lambda_1+we^{i\pi}$, which implies $ds=-dw$. Further,
\begin{align}
%\int_{DE}e^{sx}F(s,t)ds&=\int_{-\lambda_2+r}^{-\lambda_1-r}e^{sx} e^{-t[c_{1}(s+\lambda_{1})^{\alpha_1}+c^{2}(s+\lambda_2)^{\alpha_2}]}ds\\
\int_{DE}e^{sx}\overline{G}(s,t)ds=\int_{r}^{-r+\lambda_2-\lambda_1}e^{-x\lambda_1}e^{-wx} e^{t(c_1\lambda_1^{\alpha_1}+c_2\lambda_2^{\alpha_2})}e^{-t[c_1(w^{\alpha_1}e^{i\pi \alpha_1})+c_2(\lambda_2-\lambda_1+we^{i\pi})^{\alpha_2}]}dw.
\end{align}
Along FG, take $s=-\lambda_1+we^{i\pi}$, which implies $ds=-dw$ and which leads to
\begin{align}
%\int_{FG}e^{sx}F(s,t)ds&=\int_{-\lambda_1-r}^{-\lambda_2+r}e^{sx} e^{-t[c_{1}(s+\lambda_{1})^{\alpha_1}+c^{2}(s+\lambda_2)^{\alpha_2}]}ds\\
\int_{FG}e^{sx}\overline{G}(s,t)ds=-\int_{r}^{-r+\lambda_2-\lambda_1}e^{-x\lambda_1}e^{-wx} e^{t(c_1\lambda_1^{\alpha_1}+c_2\lambda_2^{\alpha_2})}e^{-t[c_1(w^{\alpha_1}e^{-i\pi \alpha_1})+c_2(\lambda_2-\lambda_1+we^{-i\pi})^{\alpha_2}]}dw.
\end{align}
Along HI, take $s= -\lambda_2+we^{i\pi}$, which implies $ds=-dw$. Hence,
\begin{align}
%\int_{HI}e^{sx}F(s,t)ds&=\int_{-\lambda_2-r}^{-R}e^{sx} e^{-t[c_{1}(s+\lambda_{1})^{\alpha_1}+c^{2}(s+\lambda_2)^{\alpha_2}]}ds\\
\int_{HI}e^{sx}\overline{G}(s,t)ds=-\int_{r}^{-\lambda_2+R}e^{-x\lambda_2}e^{-wx} e^{t(c_1\lambda^{\alpha_1}+c_2\lambda_2^{\alpha_2})}e^{-t[c_1(\lambda_1-\lambda_2+we^{-i\pi})^{\alpha_1}+c_2(w^{\alpha_2}e^{-i\pi \alpha_2})]}dw.
\end{align}
Thus,
\begin{align}\label{alonglineDE}
\int_{DE}e^{sx}\overline{G}(s,t)ds+\int_{FG}e^{sx}\overline{G}(s,t)ds &= -\int_{r}^{-r+\lambda_2-\lambda_1}e^{-x\lambda_1}e^{-wx} e^{t(c_1\lambda_1^{\alpha_1}+c_2\lambda_2^{\alpha_2})}\nonumber\\
&\times\left[e^{-t(c_1 w^{\alpha_1}\cos(\pi \alpha_1)+c_2(\lambda_2-\lambda_1+we^{-i\pi})^{\alpha_2})}2i \sin(t c_1w^{\alpha_1}\sin(\pi \alpha_1)) \right]dw.
\end{align}
Similarly,
\begin{align}\label{alonglineBC}
\int_{BC}{e^{sx}\overline{G}(s,t)ds} +\int_{HI}{e^{sx}\overline{G}(s,t)ds}
 & = \int_{r}^{-\lambda_2+R}e^{-x\lambda_2-wx+t(c_1\lambda^{\alpha_1}+c_2\lambda_2^{\alpha_2})}\nonumber\\
&\times\left[ e^{-t(c_1(\lambda_1-\lambda_2+we^{i\pi})^{\alpha_1}+c_2(w^{\alpha_2}e^{i\pi \alpha_2}))}\right.\nonumber\\ 
& \left.-e^{-t(c_1(\lambda_1-\lambda_2+we^{-i\pi})^{\alpha_1}+c_2(w^{\alpha_2}e^{-i\pi \alpha_2}))}dw \right].
\end{align} 
For $R \rightarrow \infty$ and $r \rightarrow 0$, using \eqref{PDF}, \eqref{alonglineBC}, \eqref{alonglineDE} and \eqref{resdue}, we obtain the desired result.
\end{proof}

\begin{corollary}
For the special case, $\alpha_{1}=\alpha_2=\alpha$ and $\lambda_1=0,\lambda_2=0$ with condition $c_1+c_{2}=1$, \eqref{comPDF} reduces to
\begin{equation}
g_{\alpha,0,\alpha, 0}(x,t)=\frac{1}{\pi}\int_{0}^{\infty}e^{-wx} e^{-tw^{\alpha}\cos(\pi \alpha)}\sin(tw^{\alpha}\sin(\pi \alpha))dw,
\end{equation}
which is the pdf of $\alpha$-stable subordinator (see Kumar and Vellaisamy, $2015$).
\end{corollary}
\begin{corollary}
Substituting $\alpha_1=\alpha_{2}=\alpha$, $\lambda_1=0$, $\lambda_2=\lambda >0$, $c_1=0$ and $c_2=1$ in equation then we obtain the PDF of TSS with tempering parameter $\lambda$
\begin{equation}
g_{\alpha,0,\alpha, \lambda}(x,t)=e^{-\lambda x+\lambda^{\alpha}t}\frac{1}{\pi}\int_{0}^{\infty}e^{-wx} e^{-tw^{\alpha}\cos(\pi \alpha)}\sin(tw^{\alpha}sin(\pi \alpha))dw
\end{equation}
\end{corollary}

\subsection{L\'evy density}
In this subsection, we discuss the L\'evy density for the MTSS. Here, we apply a result discussed in Bandroff-Nielsen ($2000$) for strictly increasing L\'evy processes.
\begin{proposition}
The L\'evy density denoted by $\nu_S$ for MTSS $S_{\alpha_1,\lambda_1  \alpha_2,\lambda_2}(t)$ has following forms.
When $\lambda_{1}\neq\lambda_{2}$,
\begin{align}
\nu_{S}(dx)=\frac{1}{\pi}\int_{0}^{\infty}e^{-x\lambda-wx}\left[c_1(\lambda_1-\lambda_2)^{\alpha_{1}}\sum_{k=0}^{\infty}{{{\alpha_{1}}\choose{k}}\frac{w^{k}}{(\lambda_{1}-\lambda_{2})^{k}}\sin{(\pi k)}}+c_2w^{\alpha_2}\sin{(\pi \alpha_{2})}\right]dw\nonumber\\
			 +\frac{1}{\pi}\int_{0}^{\lambda_{2}-\lambda_{1}}e^{-x\lambda-wx}\left[c_1w^{\alpha_1}\sin{(\pi \alpha_1)}+c_1(\lambda_1-\lambda_2)^{\alpha_{2}}\sum_{k=0}^{\infty}{{{\alpha_{2}}\choose{k}}\frac{w^{k}}{(\lambda_{1}-\lambda_{2})^{k}}\sin{(\pi k)}}\right]dw.
\end{align}
When $\lambda_{1}=\lambda_{2}$,
\begin{align}\label{levy for same lam}
   \nu_{S}(dx)=\frac{1}{\pi}\int_{0}^{\infty}{e^{-x\lambda-wx}\left(c_1w^{\alpha_{1}}\sin{(\pi \alpha_{1})}+c_{2}w^{\alpha_{2}}\sin{(\pi \alpha_{2})}\right)}dw.
\end{align}
\end{proposition}
\begin{proof}
Let $f(x,t)$ be the pdf for a strictly increasing L\'evy process, then the L\'evy density $\nu(dx)$ is given by
(see  Bandroff-Nielsen, et al., $2008$)
$$ \nu(dx)=\lim_{t\downarrow 0}{\frac{1}{t}f(x,t)}.$$
Using above result in \eqref{withcondpd} and \eqref{comPDF} with the help of $\lim_{t\rightarrow 0}{\frac{\sin(at)}{t}}\rightarrow a$, $a\neq 0$, gives the desired result.
\end{proof}

\begin{corollary}
Substituting $\alpha_{1}=\alpha_{2}=\alpha$ with the condition $c_{1}+c_{2}=1$ in \eqref{levy for same lam}, we obtain the L\'evy density of tempered stable subordinator, which is given by (see e.g. Cont and Tankov, $2004$, p. $115$)
$$\nu_{S}(dx) = \frac{\alpha e^{-\lambda x}}{\Gamma(1-\alpha)x^{1+\alpha}},\; x>0.$$
Further, for $\lambda=0$ and $\alpha_{1}=\alpha_{2}=\alpha$ in \eqref{levy for same lam}, the L\'evy density correspond to the $\alpha$-stable subordinator, and is given by (see e.g. Appleabum, $2009$, p. $53$)
$$\nu_{S}(dx) =\frac{\alpha}{\Gamma(1-\alpha)x^{1+\alpha}}$$
\end{corollary}
\begin{proof}
By putting $\alpha_{1}=\alpha_{2}=\alpha$ in \eqref{levy for same lam}, we obtain
$$\nu_{S}(dx) =  \frac{1}{\pi}\int_{0}^{\infty}{e^{-x\lambda-wx}w^{\alpha}\sin(\pi {\alpha}) dw}
=\frac{\alpha e^{-\lambda x}}{\Gamma(1-\alpha)x^{1+\alpha}},\; x>0.$$
Using Euler's identity $\Gamma(\alpha)\Gamma(1-\alpha)=\frac{\pi}{\sin(\pi \alpha)},$ $\forall \; \alpha \in (0,1)$, we obtain
the L\'evy density of TSS.
\end{proof}

\subsection{Moments}
In this subsection, we discuss the moments of MTSS.  We also discuss the asymptotic forms of the moments for large $t$.
The $n$-th order moment of MTSS is obtained by using $n$-th order cumulant such that 
$$k_{n}=\frac{d^{n}}{ds^{n}}K(s)|_{s=0},$$ 
where $K(s)=-t\left(c_1\left((-s+\lambda_1)^{\alpha_1}-\lambda_1^{\alpha_1}\right) + c_2\left((-s+\lambda_2)^{\alpha_2}-\lambda_2^{\alpha_2}\right)\right)$ is obtained from \eqref{PDF}. The first moment $k_{1}= \E[S_{\alpha_1,\lambda_1,\alpha_2,\lambda_2}(t)]=t(c_1 \alpha_1 {\lambda_1}^{\alpha_1-1}+c_2 \alpha_2 {\lambda_2}^{\alpha_2-1})$  and $k_{2}= {\rm Var}[S_{\alpha_1,\lambda_1,\alpha_2,\lambda_2}(t)]=t(c_1 \alpha_1(1-\alpha_{1}) {\lambda_1}^{\alpha_1-2}+c_2 \alpha_2(1-\alpha_2) {\lambda_2}^{\alpha_2-2})$. The $n$-th order cumulant is
\begin{align}\label{cumulantn-th}
k_{n} &=(-1)^{n}t [c_{1}\alpha_{1}(\alpha_{1}-1)(\alpha_{1}-2)\cdots(\alpha_{1}-n+1){\lambda_{1}}^{\alpha_{1}-n}+c_{2}\alpha_{2}(\alpha_{2}-1)(\alpha_{2}-2)\nonumber\\
&\cdots (\alpha_{2}-n+1){\lambda_{2}}^{\alpha_{1}-n} ].
\end{align}

\noindent Next, we discuss the asymptotic behavior of the $p$-th order moments of the MTSS $S_{\alpha_1,\lambda_1,\alpha_2,\lambda_2}(t)$ where $0<p<1$. 
\begin{proposition} For $0<p<1$, the asymptotic behavior of the $p$-th order moments of MTSS is given by
\begin{align}\label{frac-qthordermoment}
{\E}(S_{\alpha_1,\lambda_1,\alpha_2,\lambda_2}(t))^{p}\sim (c_{1}{\alpha_1}{\lambda_1}^{\alpha_1-1}+c_{2}{\alpha_2}{\lambda_2}^{\alpha_2-1})^{p}t^{p},\; & \mbox{as } t\rightarrow \infty,
\end{align}
with condition $c_{1}+c_{2}=1$, $c_{1}, c_{2} \geq 0$.
\end{proposition}
\begin{proof} Using the result in Kumar, et al. ($2017$),
\begin{align*}
{\E}(S_{\alpha_1,\lambda_1,\alpha_2,\lambda_2}(t))^{p} &=\frac{(-1)}{\Gamma(1-p)}\int_{0}^{\infty}{\frac{d}{ds} e^{-t[c_1((s+\lambda_1)^{\alpha_1}-\lambda_1^{\alpha_1})+c_2((s+\lambda_2)^{\alpha_2}-\lambda_2^{\alpha_2})]}s^{-p}}ds\\
&= \frac{t e^{t(c_1 \lambda_1^{\alpha_1}+c_2\lambda_2^{\alpha_2})}}{\Gamma(1-p)}\int_{0}^{\infty} s^{-p}(c_1\alpha_1(s+\lambda_1)^{\alpha_1-1}+c_2\alpha_2(s+\lambda_2)^{\alpha_2-1})\nonumber\\
&\hspace{3.4cm}\times e^{-t[c_1(s+\lambda_1)^{\alpha_1}+c_2(s+\lambda_2)^{\alpha_2}]}ds
\end{align*}
 By choosing $f(s)=c_1(s+\lambda_1)^{\alpha_1}+c_2(s+\lambda_2)^{\alpha_2}$ and $g(s)=s^{-p}(c_1\alpha_1(s+\lambda_1)^{\alpha_1-1}+c_2 \alpha_2(s+\lambda_2)^{\alpha_2-1})$, it follows
 \begin{align*}
 f(s) &=\left(c_{1}\lambda_{1}^{\alpha_1}+c_2\lambda_{2}^{\alpha_2}\right)+\left(c_{1}\alpha_{1}\lambda^{\alpha_{1}-1}+c_{2}\alpha_{2}\lambda^{\alpha_{2}-1}\right)s+...\\
& =f(0)+\sum_{k=0}^{\infty}{a_{k}s^{k+\beta}},
\end{align*}
where $f(0)=c_{1}\lambda_{1}^{\alpha_1}+c_2\lambda_{2}^{\alpha_2}$, $a_{0}=c_{1}\alpha_{1}\lambda^{\alpha_{1}-1}+c_{2}\alpha_{2}\lambda^{\alpha_{2}-1}$ and $\beta=1$. Further, 
\begin{align*}
    g(s)&=\left(c_{1}\alpha_{1}\lambda_{1}^{\alpha_1-1}+c_{2}\alpha_{2}\lambda_{2}^{\alpha_2-1}\right)s^{-p}+\left(c_{1}(\alpha_{1}-1)\lambda_{1}^{\alpha_1-2}+c_{2}(\alpha_{2}-1)\lambda^{\alpha_2-2}\right)s^{p-1}+....\\
   &= \sum_{k=0}^{\infty}{b_{k}s^{k+\gamma+1}},
\end{align*}
where $b_{0}=c_{1}\alpha_{1}\lambda_{1}^{\alpha_1-1}+c_{2}\alpha_{2}\lambda_{2}^{\alpha_2-1}$, $b_{1}=c_{1}(\alpha_{1}-1)\lambda_{1}^{\alpha_1-2}+c_{2}(\alpha_{2}-1)\lambda^{\alpha_2-2}$ and $\gamma=1-p$. Using Laplace-Erdelyi theorem (see  Erdelyi, $1956$), we have
\begin{align}\label{qthmomentapprox}
{\E}(S_{\alpha_1,\lambda_1,\alpha_2,\lambda_2}(t))^{p} \sim  \frac{t}{\Gamma(1-p)}\sum_{n=0}^{\infty}{\Gamma(n+1-p)\frac{D_{0}}{t^{n+1-p}}}, 
\end{align}
where $D_{k}$ in term of coefficients $a_{k}$ and $b_{k}$ is given by 
\begin{align}
    D_{k}=\frac{1}{a_{0}^{(n+\gamma)/ \beta}}\sum_{j=0}^{n}{b_{n-j}\sum_{i=0}^{j}{-\frac{n+\gamma}{\beta} \choose i}\frac{1}{a_{0}}{B}_{j,i}(a_{1},a_{2},...a_{j-i+1})},
\end{align}
where ${B}_{j,i}$ are the partial (or incomplete) ordinary Bell polynomials (see e.g. Andrews, $1998$). For large $t$ the dominating term is the first one in the series given in \eqref{qthmomentapprox}, which implies,
\begin{align}
    {\E}(S_{\alpha_1,\lambda_1,\alpha_2,\lambda_2}(t))^{p} \sim D_{0}t^{p},
\end{align}
where $D_{0}=(c_{1}{\alpha_1}{\lambda_1}^{\alpha_1-1}+c_{2}{\alpha_2}{\lambda_2}^{\alpha_2-1})^{p}$.
\end{proof}
\begin{remark}
For positive integer $n$, the $n$-th order moments of MTSS satisfy
$${\E}(S_{\alpha_1,\lambda_1,\alpha_2,\lambda_2}(t))^{n}=\sum_{m=1}^{n}B_{n,m}(k_{1},k_{2},...,k_{n-m+1})\sim (k_{1})^{n}\sim(c_{1}{\alpha_1}{\lambda_1}^{\alpha_1-1}+c_{2}{\alpha_2}{\lambda_2}^{\alpha_2-1})^{n}t^{n} \; \;{\rm as} \; t\rightarrow \infty,
$$
where $B_{n,m}$ partial exponential polynomials (see e.g. Andrew $1998$). For more information on cumulants, Bell polynomials and moments see e.g. Rota and Shen (2000) and Smith (1995).
\end{remark}
\begin{remark}
By taking $\alpha_1,\; \alpha_2=\alpha,$ \; $\lambda_1,\; \lambda_2=\lambda$ with condition $c_1+c_2=1$, $c_{1}, c_{2} \geq 0$ in \eqref{frac-qthordermoment}, we obtain the asymptotic behaviour of the $p$-th order moments of TSS $S_{\lambda, \alpha}(t),$ and which is given by
$$
{\E}(S_{\lambda,\alpha}(t))^{p} \sim (\alpha {\lambda}^{\alpha-1}t)^{p},\; \quad \mbox{as } t\rightarrow \infty.$$
Similarly, the asymptotic behavior of $n$-th order moment for TSS is
$$
{\E}(S_{\lambda,\alpha}(t))^{n} \sim (\alpha {\lambda}^{\alpha-1}t)^{n},\; \quad \mbox{as } t\rightarrow \infty.$$
\end{remark}

\subsection{Governing fractional Fokker-Planck-Kolmogorov (FPK) equations}
In this subsection, the governing fractional FPK type equation for MTSS is discussed. We recall the LT denoted by $\mathcal{L}_t$ with respect to time variable $t$ of shifted fractional Riemann-Liouville (RL) derivatives, which is given by (see e.g. Gorenflo and Mainardi, 1997; Beghin, 2015),
\begin{equation}\label{FRLD}
\mathcal{L}_{t} \left(c+\frac{\partial}{\partial t}\right)^{\nu}f(x,t)=(c+s)^{\nu}\mathcal{L}_{t}{f(x,t)}-(c+s)^{\nu-1}f(x,0),\; s>0. 
\end{equation}
The shifted fractional RL derivative can be defined as in Beghin, 2015, see also the approach discussed in  Leonenko, et al. $2019$.
We also recall the definition of generalized Mittag-Leffler function (see e.g. Prabhakar, 1971),
$$M^{r}_{p,q}(z)= \sum_{k=0}^{\infty}\frac{(r)_{n}}{\Gamma(pn+q)}\frac{z_{n}}{n!}$$
where p,q,r $\in \mathbb{C}$ with $\mathcal{R}(q)>0$ and $(r)_{n}=\frac{\Gamma(r+n)}{\Gamma(r)}$ is Pochhammer symbol. The LT  $F(s)=\frac{s^{pr-q}}{(s^{p}+a)^{r}}$ has the inverse LT (see e.g. Monje, 2010)
\begin{equation}\label{inverseLTM}
\mathcal{L}^{-1}[F(s)]=t^{q-1}M^{r}_{p,q}(-at^{p}).
\end{equation}
\begin{proposition}
The pdf $g_{\alpha_1, \lambda_1, \alpha_2, \lambda_2}(x,t) \equiv G(x,t)$ of the MTSS satisfies the following fractional partial differential equation (FPDE)
\begin{align}
\frac{\partial}{\partial t}G(x,t) & =-c_{1}\left(\lambda_1+\frac{\partial}{\partial x}\right)^{\alpha_1}G(x,t)-c_{2}\left(\lambda_2+\frac{\partial}{\partial x}\right)^{\alpha_2}G(x,t)+\lambda_1^{\alpha_1}{c_1}G(x,t)+\lambda_2^{\alpha_1}{c_2}G(x,t),
\end{align}
with initial conditions
\begin{equation}
\begin{cases}
    G(x,0)= \delta(x)\\
     G(0,t)= 0.
\end{cases}
\end{equation}
\end{proposition}
\begin{proof} Using \eqref{PDF}, 
$$\mathcal{L}_{x}\left(g_{\alpha_1, \lambda_1, \alpha_2, \lambda_2}(x,t)\right)=\mathcal{L}_{x}\left(G(x,t)\right)= e^{-t\left(c_1\left((s+\lambda_1)^{\alpha_1}-\lambda_1^{\alpha_1}\right) + c_2\left((s+\lambda_2)^{\alpha_2}-\lambda_2^{\alpha_2}\right)\right)}=\overline{G}(s,t).$$
Differentiating with respect to $t$ yields
\begin{align}
\frac{\partial}{\partial t}{\overline{G}}(s,t) &=-\left[c_1\left((s+\lambda_1)^{\alpha_1}-\lambda_1^{\alpha_1}\right) + c_2\left((s+\lambda_2)^{\alpha_2}-\lambda_2^{\alpha_2}\right)\right]\overline{G}(s,t)\nonumber\\
&= -\left[c_{1}(s+\lambda_{1})^{\alpha_1}\overline{G}(s,t)-c_{1}(s+\lambda_{1})^{\alpha_1-1}G(0,t)\right]-\left[c_{2}(s+\lambda_{2})^{\alpha_2}\overline{G}(s,t)-c_{1}(s+\lambda_{1})^{\alpha_1-1}G(0,t)\right]\nonumber\\
&+c_{1}{\lambda_{1}}^{\alpha_{1}}\overline{G}(s,t)+c_{2}{\lambda_{2}}^{\alpha_{2}}\overline{G}(s,t)-c_{1}(s+\lambda_{1})^{\alpha_1-1}G(0,t)-c_{1}(s+\lambda_{1})^{\alpha_1-1}G(0,t).\nonumber
\end{align}
Taking the inverse LT on both sides and using equation \eqref{FRLD} and apply the initial conditions, we obtain the desired result.

\end{proof}
\subsection{Asymptotic form of potential density}
In this subsection, we discuss the asymptotic behavior of the potential density at $0$ (respectively at $\infty$) for MTSS. The potential measure of a subordinator $S(t)$ is defined by (see Sikic et al., $2005$) 
\begin{align}
    V(A)=\E{\int_{0}^{\infty}{1_{(S_{t}\in A)} dt }}.
\end{align}
The LT of the measure $V$ is given by
\begin{align}
  \overline{V}(s) = \E{\int_{0}^{\infty}{\exp(- s S_{t})}dt}=\frac{1}{\phi{(s)}}.
\end{align}
Note that potential measure represent the expected time the subordinator spent in the set $A$.
\begin{proposition} Let $v$ be the potential density of the MTSS. For any $\alpha_{1},\alpha_{2} \in(0,1]$, we have
\begin{equation}\label{potentialdensity}
v(x)\sim
\left\{
		\begin{array}{lll}
			 \frac{x^{\alpha_{1}+\alpha_{2}-\min(\alpha_{1},\alpha_{2})-1}}{\Gamma(\min(\alpha_{1},\alpha_{2}))(c_1x^{\alpha_2-\min(\alpha_{1},\alpha_{2})}+c_2x^{\alpha_{1}-\min(\alpha_{1},\alpha_{2})})}, \;&\mbox{as}\;x\rightarrow 0,\\
			\frac{1}{c_{1}\alpha_{1}{\lambda_{1}}^{\alpha_{1}-1}+c_{2}\alpha_{2}{\lambda_{2}}^{\alpha_{2}-1}},\; \lambda_1,\lambda_2 >0, &\mbox{as }x\rightarrow \infty,\\
			\frac{x^{\alpha_{1}+\alpha_{2}-\min(\alpha_{1},\alpha_{2})-1}}{\Gamma(\min(\alpha_{1},\alpha_{2}))(c_1x^{\alpha_{2}-\min(\alpha_{1},\alpha_{2})}+c_2x^{\alpha_{1}-\min(\alpha_{1},\alpha_{2})})}, \;\lambda_1, \lambda_2=0   &\mbox{as}\;x\rightarrow \infty.\\
		\end{array}
	\right.
\end{equation}
\end{proposition}
\begin{proof}
We apply the Tauberian theorem  (see e.g. Bertoin, 1996, p. 10), which connects the asymptotic form of a Laplace transform with its inverse Laplace transform for a function. We have
$$ \overline{V}(s)=\frac{1}{\phi(s)}\sim\frac{s^{-1}}{(\alpha_1{\lambda_1}^{\alpha_1-1}+\alpha_2{\lambda_2}^{\alpha_2-1})},\; \mbox{as}\;s\rightarrow 0. $$
Similarly, for $\lambda_{1},\lambda_{2}>0,$
$$ \overline{V}(s) \sim \frac{1}{c_{1}s^{\alpha_{1}}+c_{2}s^{\alpha_{2}}},\; \mbox{as}\;s\rightarrow \infty.$$
Applying the Tauberian theorem at $ x \rightarrow 0$ (respectively at $x \rightarrow \infty$) gives the desired result.
\end{proof}
\begin{remark}
By substituting  $\alpha_1=\alpha_2=\alpha$ and $\lambda_1=\lambda_2=\lambda$, with the condition $c_1+c_2=1$ in \eqref{potentialdensity}, we obtain the asymptotic behavior of potential density for TSS, such that 
\begin{align}
    v(x)\sim
\left\{
		\begin{array}{ll}
			 \frac{\lambda^{1-\alpha}}{\alpha}, \;&\mbox{as}\;x\rightarrow \infty,\\
			 \frac{x^{\alpha-1}}{\Gamma(\alpha)} &\mbox{as}\;x\rightarrow 0.
		\end{array}
	\right.
\end{align}
\end{remark}
\subsection{Inverse of mixture tempered stable subordinator (IMTSS)}
Let $E_{\alpha_1,\lambda_1,\alpha_2, \lambda_2}(t)$ be the right continuous inverse of MTSS $S_{\alpha_1,\lambda_1,\alpha_2, \lambda_2}(t)$, defined by
$$  E_{\alpha_1,\lambda_1,\alpha_2, \lambda_2}(t)=\inf\{u >0:\; S_{\alpha_{1},\lambda_{1},\alpha_{2},\lambda_{2}}(u)>t\}.$$
The process $E_{\alpha_1,\lambda_1,\alpha_2, \lambda_2}(t)$ is called inverse of mixture tempered stable (IMTS) subordinator. This is also called the first-exist time. Since MTSS is a strictly increasing L\'evy process, the sample paths of $E_{\alpha_1,\lambda_1,\alpha_2, \lambda_2}(t)$ are almost surely continuous and are constant over the intervals where $S_{\alpha_1,\lambda_1,\alpha_2, \lambda_2}(t)$ have jumps. Let $h_{\alpha_1,\lambda_1,\alpha_2, \lambda_2}(t)$ be the pdf of IMTSS, then the Laplace transform $\tilde{h}_{\alpha_1,\lambda_1,\alpha_2, \lambda_2}(x,s)$ of the density  $h_{\alpha_1,\lambda_1,\alpha_2, \lambda_2}(t)$ with respect to the time variable $t$ is given by  ( see Meerschaert and Scheffler, $2008$),
\begin{align}\label{PDF_LT_IMTSS} 
 \tilde{h}_{\alpha_1,\lambda_1,\alpha_2, \lambda_2}(x,s)= \frac{\phi(s)}{s}e^{-x \phi(s)},
\end{align}
where 
\begin{align}\label{phi}
\phi(s)= c_1\left((s+\lambda_1)^{\alpha_1}-\lambda_1^{\alpha_1}\right) + c_2\left((s+\lambda_2)^{\alpha_2}-\lambda_2^{\alpha_2}\right).
\end{align}
\begin{proposition}
The pdf $h_{\alpha_1,\lambda_1,\alpha_2, \lambda_2}(x,t)\equiv H(x,t)$ of IMTSS governs the following time-fractional differential equation
\begin{align}\label{Gover_IMTSS}
\frac{\partial}{\partial x}H(x,t) & =-c_{1}\left(\lambda_1+\frac{\partial}{\partial t}\right)^{\alpha_1}H(x,t)-c_{2}\left(\lambda_2+\frac{\partial}{\partial t}\right)^{\alpha_2}H(x,t)+\lambda_1^{\alpha_1}{c_1}H(x,t)+\lambda_2^{\alpha_1}{c_2}H(x,t)\nonumber \\
 &-c_1 t^{-\alpha_1}M_{1,1-\alpha_1}^{1-\alpha_1}(-\lambda_{1}t)\delta(x) -c_2 t^{-\alpha_2}M_{1,1-\alpha_2}^{1-\alpha_2}(-\lambda_{2}t)\delta(x),
\end{align}
with $H(x,0)= \delta(x)$.
\end{proposition}
\begin{proof} Using \eqref{PDF_LT_IMTSS}, 
\begin{align*}
\mathcal{L}_{t}\left(h_{\alpha_1, \lambda_1, \alpha_2, \lambda_2}(x,t)\right) &= \frac{c_1\left((s+\lambda_1)^{\alpha_1}-\lambda_1^{\alpha_1}\right) + c_2\left((s+\lambda_2)^{\alpha_2}-\lambda_2^{\alpha_2}\right)}{s}e^{-t\left(c_1\left((s+\lambda_1)^{\alpha_1}-\lambda_1^{\alpha_1}\right) + c_2\left((s+\lambda_2)^{\alpha_2}-\lambda_2^{\alpha_2}\right)\right)}\\
&=\overline{H}(x,s),
\end{align*}
which implies
\begin{align}
\frac{\partial}{\partial x}{\overline{H}}(x,s) &=-\left[c_1\left((s+\lambda_1)^{\alpha_1}-\lambda_1^{\alpha_1}\right) + c_2\left((s+\lambda_2)^{\alpha_2}-\lambda_2^{\alpha_2}\right)\right]\overline{H}(x,s)\nonumber\\
&= -\left[c_{1}(s+\lambda_{1})^{\alpha_1}\overline{H}(x,s)-c_{1}(s+\lambda_{1})^{\alpha_1-1}H(x,0)\right]-\left[c_{2}(s+\lambda_{2})^{\alpha_2}\overline{H}(x,t)-c_{1}(s+\lambda_{1})^{\alpha_1-1}H(x,0)\right]\nonumber\\
&+c_{1}{\lambda_{1}}^{\alpha_{1}}\overline{H}(x,s)+c_{2}{\lambda_{2}}^{\alpha_{2}}\overline{H}(x,s)-c_{1}(s+\lambda_{1})^{\alpha_1-1}H(x,0)-c_{1}(s+\lambda_{1})^{\alpha_1-1}H(x,0).\nonumber
\end{align}
Taking the inverse LT on both sides and using equation \eqref{FRLD}, we obtain
\begin{align}\label{Differential eqn}
\frac{\partial}{\partial x}{H(x,t)}&=-c_{1}\left(\lambda_{1}+\frac{\partial}{\partial t}\right)^{\alpha_{1}}H(x,t)-c_{2}\left(\lambda_{2}+\frac{\partial}{\partial t}\right)^{\alpha_{2}}H(x,t)+c_{1}{\lambda_{1}}^{\alpha_{1}}H(x,t)+c_{2}{\lambda_{2}}^{\alpha_{2}}H(x,t)\nonumber\\
&-\mathcal{L}^{-1}\left[c_{1}(s+\lambda_{1})^{\alpha_1-1}\right]H(x,0)-\mathcal{L}^{-1}\left[c_{1}(s+\lambda_{1})^{\alpha_1-1}\right]H(x, 0).
\end{align}
In \eqref{inverseLTM}, by taking $p=1, q = 1-{\alpha}, r = 1$ and $ a= \lambda$, yields
\begin{align}\label{mittag-leffler}
\mathcal{L}^{-1}\left[\frac{1}{(s+\lambda)^{1-\alpha}}\right]=t^{-\alpha} M^{1-\alpha}_{1, 1-\alpha}(-\lambda t).
\end{align}
Using \eqref{Differential eqn} and \eqref{mittag-leffler}, yields the desired result.
\end{proof}

\subsection{Asymptotic form of the renewal function}
%Let $E_{\alpha_1,\lambda_1,\alpha_2, \lambda_2}(t)=\inf\{u >0:\; S_{\alpha_{1},\lambda_{1},\alpha_{2},\lambda_{2}}(u)>t\}$ be the right continuous inverse of the MTSS (IMTSS). This is also called the first-exit time.
The renewal function is given by $U(t)=\mathbb{E}(E_{\alpha_1,\lambda_1,\alpha_2,\lambda_2}(t))$.
The Laplace transform (LT) of $U(t)$ is $\overline{U}(s)=\frac{1}{s\phi(s)}$ (see Leonenko et. al., 2014).\\
%where 
%\begin{align}\label{phi}
%\phi(s)= c_1\left((s+\lambda_1)^{\alpha_1}-\lambda_1^{\alpha_1}\right) + c_2\left((s+\lambda_2)^{\alpha_2}-\lambda_2^{\alpha_2}\right).
%\end{align}
Next, we discuss the asymptotic form of the renewal function.
\begin{proposition} The renewal function $U(t)$ has following asymptotic form,
\begin{equation}
 	U(t){\sim}
\left\{
	\begin{array}{lll}
		\frac{t^{\alpha_1+\alpha_2-\min(\alpha_{1},\alpha_{2})}}{\Gamma(1+\min(\alpha_{1},\alpha_{2}))(c_1t^{\alpha_2-\min(\alpha_{1},\alpha_{2})}+c_2t^{\alpha_1-\min(\alpha_{1},\alpha_{2})})}, \;&\mbox{as}\;t\rightarrow 0,\nonumber\\
		\frac{t}{(c_{1}\alpha_1{\lambda_1}^{\alpha_{1}-1}+c_{2}\alpha_2{\lambda_2}^{\alpha_2-1})},\; \lambda_1,\lambda_2 >0,\;\;   &\mbox{as}\;t\rightarrow \infty,\nonumber\\
			\frac{t^{\alpha_1+\alpha_2-\min(\alpha_{1},\alpha_{2})}}{\Gamma(1+\min(\alpha_{1},\alpha_{2}))(c_1t^{\alpha_2-\min(\alpha_{1},\alpha_{2})}+c_2t^{\alpha_1-\min(\alpha_{1},\alpha_{2})})},\; \lambda_1=\lambda_2 =0, & \mbox{as } t\rightarrow \infty. \\
    \end{array}
\right.
\end{equation}
\end{proposition}
\begin{proof}
Using the Tauberian theorem (see e.g. Bertoin 1996, p. 10), which says that $U(t)\sim t^{p}\frac{l(t)}{\Gamma(p+1)}$ as ${t\to\infty}$ (respectively at $0$)  is equivalent to $\overline{U}(s){\sim}s^{-1-p}l(\frac{1}{s})$ as ${s\to 0}$ (respectively at $\infty$), where $l : (0,\infty)\to (0.\infty)$ is a slowly varying function at $0$ (respectively at $\infty$).
When ${s\to 0}$, the Laplace exponent behaves as
\begin{align}\label{phiatzero}
\phi(s){\sim}s(c_{1}\alpha_1{\lambda_1}^{\alpha_1-1}+c_{2}\alpha_2{\lambda_2}^{\alpha_2-1}),
\end{align}
and hence 
$$\overline{U}(s)=\frac{1}{s\phi(s)}=\frac{1}{s(c_1\left((s+\lambda_1)^{\alpha_1}-\lambda_1^{\alpha_1}\right) + c_2\left((s+\lambda_2)^{\alpha_2}-\lambda_2^{\alpha_2}\right))}{\sim}\frac{s^{-2}}{(\alpha_1{\lambda_1}^{\alpha_1-1}+\alpha_2{\lambda_2}^{\alpha_2-1})},$$
which further implies that the renewal function has the asymptotic form
\begin{equation}
U(t){\sim}\frac{t}{(c_{1}\alpha_1{\lambda_1}^{\alpha_{1}-1}+c_{2}\alpha_2{\lambda_2}^{\alpha_2-1})},\; \lambda_1,\lambda_2 >0, \;\mbox{as}\;t\rightarrow \infty.
\end{equation}
For $\lambda_1=\lambda_2=0$, we have 
\begin{equation*}
U(t){\sim}\frac{t^{\alpha_{1}+\alpha_{2}- \min(\alpha_{1},\alpha_{2})}}{\Gamma(1+\min(\alpha_{1},\alpha_{2}))(c_1t^{\alpha_2-\min(\alpha_{1},\alpha_{2})}+c_2t^{\alpha_1- \min(\alpha_1, \alpha_2)})}, \;\mbox{as}\;t\rightarrow \infty.
\end{equation*}
Moreover,
\begin{align}\label{phiatinfty}
\phi(s){\sim}c_1 s^{\alpha_1}+c_2 s^{\alpha_2}\; {\rm as}\; {s\to\infty},
\end{align}
and hence
\begin{equation}
\overline{U}(s){\sim}\frac{1}{s^{1+ \min(\alpha_1.\alpha_2)}(c_2 s^{\alpha_2-\min(\alpha_1.\alpha_2)}+c_1s^{\alpha_1-\min(\alpha_1.\alpha_2)})}, \;\mbox{as}\;s\rightarrow \infty,
\end{equation}
where $l(s)=\frac{s^{\alpha_1+\alpha_2-2 \min(\alpha_1.\alpha_2)}}{(c_2s^{\alpha_1- \min(\alpha_1.\alpha_2)}+c_1s^{\alpha_2-\min(\alpha_1.\alpha_2)})}$ is slowly varying function at ${\infty}$ and hence the renewal function
\begin{align}
U(t){\sim}\frac{t^{\alpha_1+\alpha_2-\min(\alpha_{1},\alpha_{2})}}{\Gamma(1+\min(\alpha_{1},\alpha_{2}))(c_1t^{\alpha_2-\min(\alpha_{1},\alpha_{2})}+c_2t^{\alpha_1-\min(\alpha_{1},\alpha_{2})})}, \;\mbox{as}\;t\rightarrow 0.
\end{align}
\end{proof}
\begin{remark}
Substitute  $\alpha_1=\alpha_2=\alpha$ and $\lambda_1=\lambda_2=\lambda$ with condition $c_1+c_2=1$ in \eqref{momentasym}. which gives the asymptotic behaviour of renewal function corresponding to TSS $S_{\alpha, \lambda}(t)$  (see Leonenko et.al., $2014$),
\begin{align}
    U(t)\sim
\left\{
		\begin{array}{ll}
			 t\frac{\lambda^{\alpha-1}}{\alpha}, \;&\mbox{as}\;t\rightarrow \infty\\
			 \frac{t^{\alpha}}{\Gamma(1+\alpha)} &\mbox{as}\;t\rightarrow 0.\\
		\end{array}
	\right.
\end{align}
\end{remark}
\noindent Next, we discuss the asymptotic behavior of $q$-th order moments  $M_q(t)=\mathbb{E}(E_{\alpha_1,\lambda_1,\alpha_2,\lambda_2}(t))^{q}, q>0,$ of $E_{\alpha_1,\lambda_1,\alpha_2,\lambda_2}(t)$. The LT of $M_q(t)$ is given by $\overline{M}_q(s)=\frac{\Gamma(1+q)}{s(\phi(s))^{q}}$ (see e.g. Veillette and Taqqu, 2010a, 2010b; Kumar, et al., $2017$), where $\phi(s)$ is the Laplace exponent given in \eqref{phi}.
Again using Tauberain theorem, we have the following asymptotic behavior for $M_q(t)$
\begin{align}\label{momentasym}
M_q(t) =
\left\{
	\begin{array}{lll}
		\frac{t^{q({\alpha_1+\alpha_2-\min(\alpha_{1},\alpha_{2})})}\Gamma(1+q)}{\Gamma(1+q \min(\alpha_{1},\alpha_{2}))(c_1t^{\alpha_2-\min(\alpha_{1},\alpha_{2})}+c_2t^{\alpha_1-\min(\alpha_{1},\alpha_{2})})^{q}}, & \mbox{as } t\rightarrow 0, \\
		\frac{t^{q}}{(c_1\alpha_1{\lambda_1}^{\alpha_1-1}+c_2\alpha_2{\lambda_2}^{\alpha_2-1})^{q}},\; \lambda_1,\lambda_2 >0, & \mbox{as } t\rightarrow \infty,\\
		\frac{t^{q({\alpha_1+\alpha_2- \min(\alpha_{1},\alpha_{2})})}\Gamma(1+q)}{\Gamma(1+q \min(\alpha_{1},\alpha_{2}))(c_1t^{\alpha_2-\min(\alpha_{1},\alpha_{2})}+c_2t^{\alpha_1- \min(\alpha_{1},\alpha_{2})})^{q}},\; \lambda_1=\lambda_2 =0, & \mbox{as } t\rightarrow \infty. \\
    \end{array}
\right.
\end{align}
\begin{corollary}
By taking  $\alpha_1=\alpha_2=\alpha$ and $\lambda_1=\lambda_2=\lambda$ with condition $c_1+c_2=1$, $c_{1}, c_{2} \geq 0$ in \eqref{momentasym}, which gives the asymptotic behaviour of $M_q(t)$ for TSS $S_{\alpha, \lambda}(t)$ such that
\begin{align}
  M_{q}(t)  \sim 
\left\{
		\begin{array}{lll}
			 \frac{\Gamma(1+q)}{\Gamma(1+q\alpha)}t^{q\alpha}, \;&\mbox{as}\;t\rightarrow 0,\\
			 \frac{\lambda^{q(1-\alpha)}}{\alpha^{q}}t^{q}, \; \lambda>0, &\mbox{as}\;t\rightarrow \infty,\\
			 \frac{\Gamma(1+q)}{\Gamma(1+q\alpha)}t^{q\alpha}, \; \lambda=0, &\mbox{as}\;t\rightarrow \infty.\\
	   \end{array}
	\right.
\end{align}
\end{corollary}

\subsection{Simulation of MTSS sample trajectories and its inverse}
In this subsection, we discuss the algorithm to simulate the sample trajectories of MTSS and its inverse. The algorithm for generating the sample trajectories of MTSS are as follows:\\
\textbf{Step 1}: fix the values of parameters; generate independent, and uniformly distributed in
$[0,1]$ rvs $U, V$;\\
\textbf{Step 2}: generate the increments of the $\alpha$-stable subordinator $S_{\alpha}(t)$ (see Cahoy et. al., $2010$) with pdf $f_{\alpha}(x,t)$, using the relationship $S_{\alpha}(t+dt) - S_{\alpha}(t) \stackrel{d} = S_{\alpha}(dt) \stackrel{d} = (dt)^{1/\alpha}S_{\alpha}(1)$, where
\begin{equation*}
S_{\alpha}(1) \stackrel{d} = \frac{\sin(\alpha\pi U)[\sin((1-\alpha)\pi U)]^{1/\alpha-1}}{[\sin(\pi U)]^{1/\alpha} |\ln V|^{1/\alpha-1}};\;
\end{equation*}
\textbf{Step 3}: for generating the increments of TSS $S_{\alpha, \lambda}(t)$ with pdf $f_{\alpha,\lambda}(x,t)$, we use the following steps called ``acceptance-rejection method",
\begin{enumerate}[(a)]
\item generate the stable random variable $S_{\alpha}(dt)$;
\item generate uniform $(0, 1)$ rv $W$ (independent from $S_{\alpha}$);
\item if $W \leq e^{-\lambda S_{\alpha}(dt)}$, 
then $ S_{\alpha,\lambda}(dt) = S_{\alpha}(dt)$ (``accept"); otherwise  go back to $(a)$ (``reject").\\
Note that, here we used \eqref{ts-density}, which implies $\frac{f_{\alpha,\lambda}(x, dt)}{c f_{\alpha}(x,dt)} = e^{-\lambda x}$ for $c = e^{\lambda^{\alpha}dt}$ and the ratio is bounded between 0 and 1;
\end{enumerate}
\textbf{Step 4}: cumulative sum of increments gives the TSS $S_{\alpha, \lambda}(t)$ sample trajectories;\\
\textbf{Step 5}: generate $S_{\alpha_1, \lambda_1}(c_1 t)$, $S_{\alpha_2, \lambda_2}(c_2 t)$ and add these to get the MTSS, see \eqref{alternative-rep}. The inverse MTSS sample trajectories are obtained by reversing the axis.

%\begin{figure*}[!ht]
%\begin{center}
%\includegraphics[width=.3 \columnwidth]{MTSS.png}\hspace{0.99cm}
%\includegraphics[width=.3 \columnwidth]{IMTSS.png}
%\caption{Sample-paths of MTSS}
%\end{center}
%\end{figure*}
\begin{figure}[!ht]
\centering
\begin{minipage}{.5\textwidth}
  \centering
  \includegraphics[width=.9\linewidth]{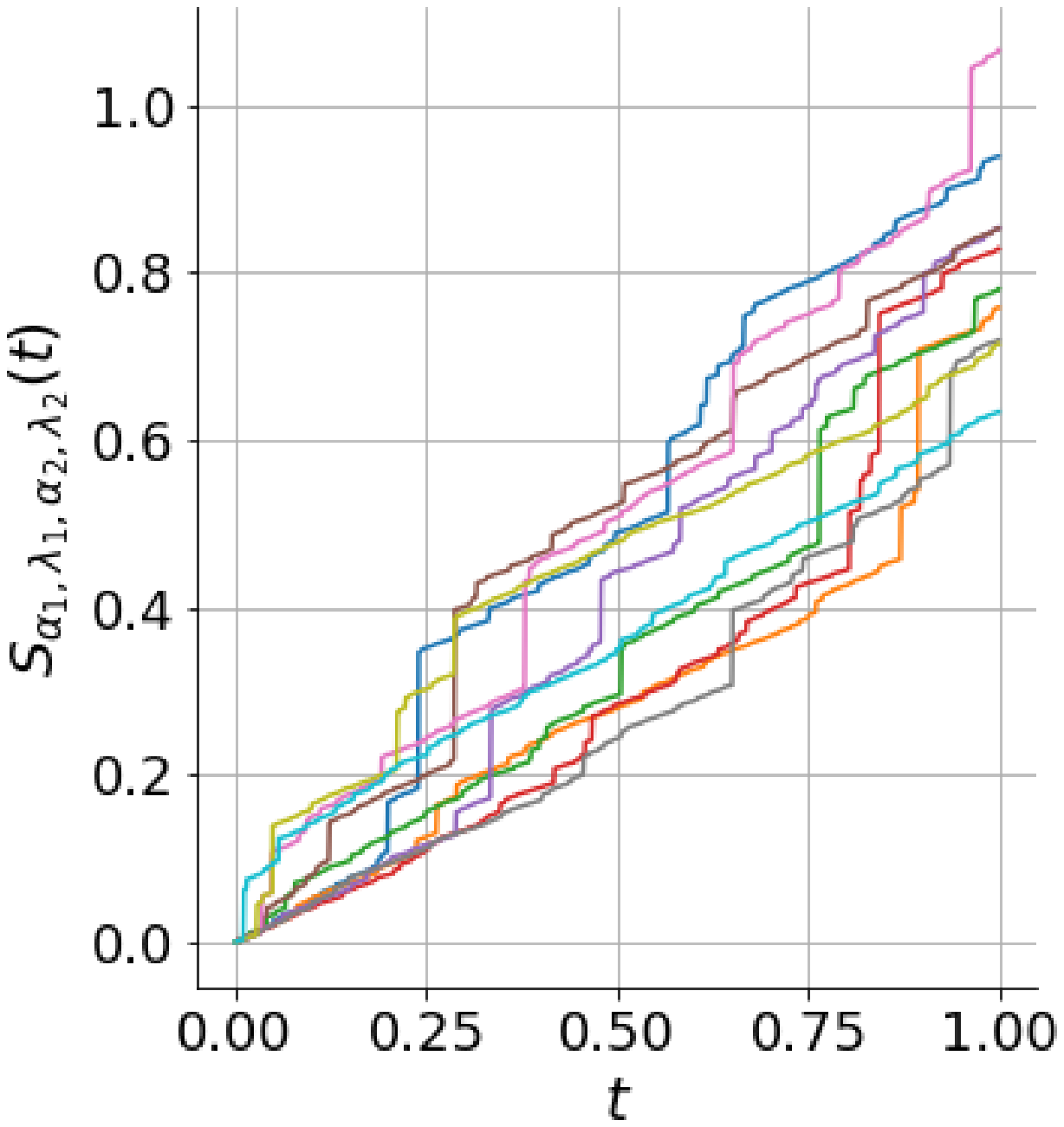}
  \caption{MTSS}
  \label{fig:test1}
\end{minipage}%
\begin{minipage}{.5\textwidth}
  \centering
  \includegraphics[width=.9\linewidth]{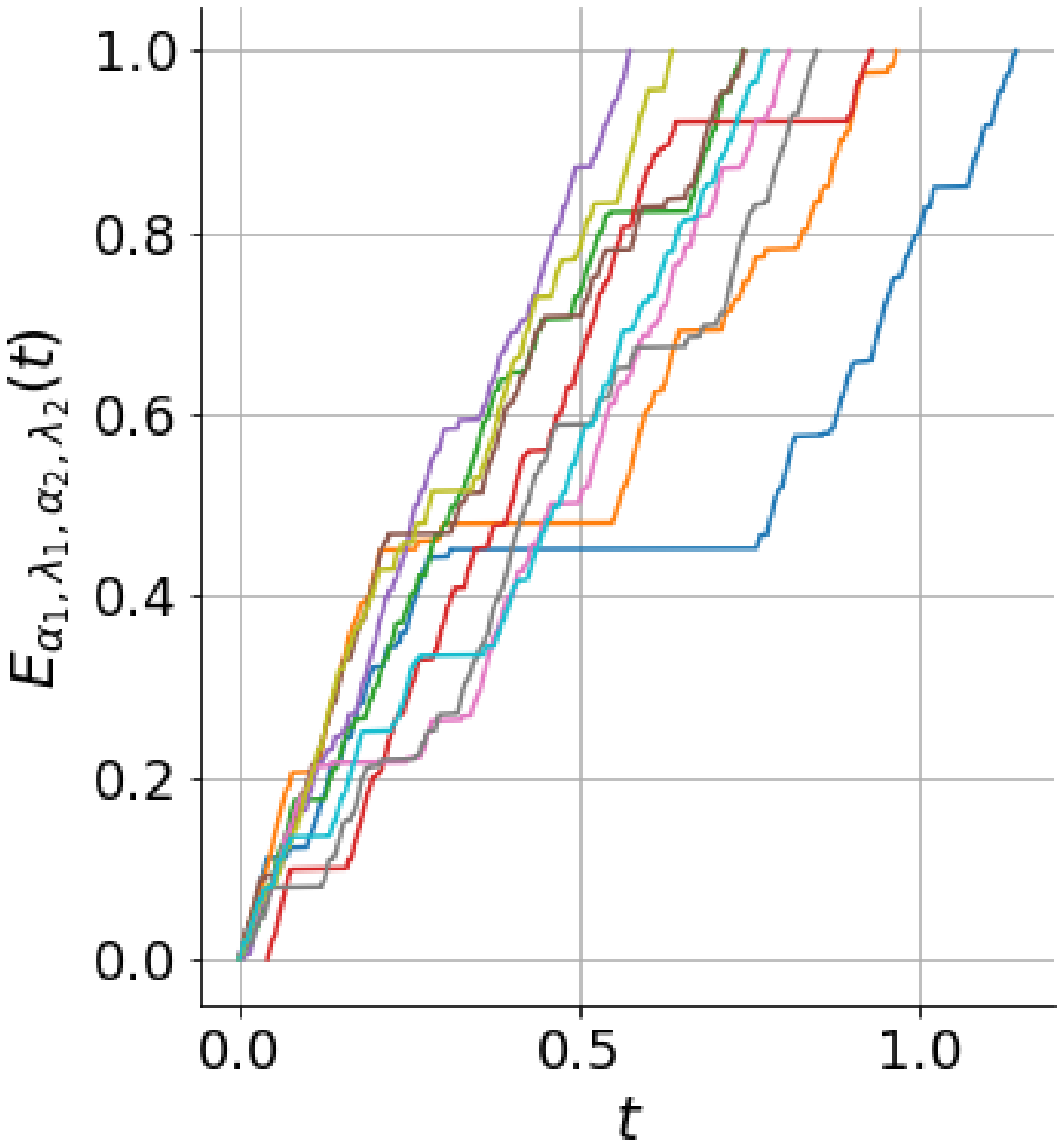}
  \caption{Inverse MTSS}
  \label{fig:test2}
\end{minipage}
\end{figure}

\section{$N$-th order mixtures of tempered stable subordinators}
In this subsection, we generalize the MTSS by taking $n$ mixtures of TSS.
We define the $n$-th order mixtures of TSS as a L\'evy process with LT:
\begin{align}
{\E}\left(e^{-sS_{\alpha_1, \lambda_1, \alpha_2, \lambda_2,\ldots,\alpha_n, \lambda_n}(t)}\right) = e^{-t\sum_{i=1}^{n}{c_{i}((s+\lambda_{i})^{\alpha_{i}}-{\lambda_{i}}^{\alpha_{i}})}},\;s>0,
\end{align}
where $c_{i}\geq{0}$ and $\sum_{i=1}^{n}{c_{i}}=  1.$
The alternative representation of $n$-th order MTSS is given by
$$ S_{\alpha_{1}, \lambda_{1},\alpha_{2}, \lambda_{2},\dots ,\alpha_{n}, \lambda_{n}} (t) =\sum_{i=1}^{n}S_{\alpha_{i},\lambda_{i}} (c_i t),
$$
with the conditions $c_{i} \geq 0$, $\sum_{i=1}^{n}{c_{i}}=1$. Using similar approaches as in previous subsections, we can obtain analogues results for $n$-th order mixtures of TSS. The pdf of $n$-th order mixtures of TSS is difficult to obtain using complex inversion. For $0<p<1$, the asymptotic behavior of the $p$-th order moments of $n$-th order mixtures of TSS is given by
\begin{align}\label{frac-qthordermoment}
{\E}\left(S_{\alpha_1, \lambda_1, \alpha_2, \lambda_2,\ldots,\alpha_n, \lambda_n}(t)\right)^{p}\sim \left(\sum_{i=1}^{n}{(c_{i}{\alpha_i}{\lambda_i}^{\alpha_i-1})}\right)^{p}t^{p},\; & \mbox{as } t\rightarrow \infty.
\end{align}
\noindent The generalized  PDF $g_{\alpha_1, \lambda_1, \alpha_2,\lambda_2 \dots,\alpha_{n}, \lambda_{n}}(x,t)$ for $n$-th order  mixtures of TSS satisfies the following FPDE with condition $g_{\alpha_1, \lambda_1, \alpha_2,\lambda_2 \dots,\alpha_{n}, \lambda_{n}}(0,t)=0$,
\begin{align}
\frac{\partial}{\partial t}g_{\alpha_1, \lambda_1, \alpha_2,\lambda_2 \dots,\alpha_{n}, \lambda_{n}}(x,t)  =-\sum_{i=1}^{n}c_{i}\left(\lambda_i+\frac{\partial}{\partial x}\right)^{\alpha_i}g_{\alpha_1, \lambda_1, \alpha_2,\lambda_2 \dots,\alpha_{n}, \lambda_{n}}(x,t) + \left(\sum_{i=1}^{n}{\lambda_{i}^{\alpha_i}{c_i}}\right)g_{\alpha_1, \lambda_1, \alpha_2,\lambda_2 \dots,\alpha_{n}, \lambda_{n}}(x,t).\nonumber\\
\end{align}          
Further, the asymptotic behaviour of $v(x)$ for $n$-th order mixtures of TSS is obtained in same manner and which is given by
\begin{equation}
 	v(x){\sim}
\left\{
	\begin{array}{lll}
		\frac{x^{\sum_{i=1}^{n}{{{\alpha_{i}}}-(n-1)\min(\alpha_{1},\alpha_{2},\dots,\alpha_{n})}-1}}{\Gamma(\min(\alpha_{1},\alpha_{2},\dots,\alpha_{n}))\left(\sum_{i=1}^{n}{c_{i}x^{\sum_{j\neq i}^{n}{{{\alpha_{i}}}-(n-1)\min(\alpha_{1},\alpha_{2},\dots,\alpha_{n})}}}\right)}, \;&\mbox{as}\;x\rightarrow 0,\nonumber\\
		\frac{1}{\left(\sum_{i=0}^{n}{c_{i}\alpha_{i}{\lambda_{i}}^{\alpha_{i}-1}}\right)},\; \lambda_1,\lambda_2,\ldots,\lambda_{n} >0,\;\;   &\mbox{as}\;x\rightarrow \infty,\nonumber\\
		\frac{x^{\sum_{i=1}^{n}{{{\alpha_{i}}}-(n-1)\min(\alpha_{1},\alpha_{2},\dots,\alpha_{n})}-1}}{\Gamma(\min(\alpha_{1},\alpha_{2},\dots,\alpha_{n}))\left(\sum_{i=1}^{n}{c_{i}x^{\sum_{j\neq i}^{n}{{{\alpha_{i}}}-(n-1)\min(\alpha_{1},\alpha_{2},\dots,\alpha_{n})}}}\right)},	\; \lambda_1,\lambda_2,\dots,\lambda_{n} =0, & \mbox{as } x\rightarrow \infty. \\
    \end{array}
\right.
\end{equation}
The asymptotic behavior of renewal function corresponding to the $n$-th order compositions of TSS is given by
\begin{align}
 	U(t){\sim}
\left\{
\begin{array}{lll}
 \frac{t^{\left(\sum_{i=1}^{n}{{\alpha_{i}}}-(n-1)\min(\alpha_{1},\alpha_{2},\dots,\alpha_{n})\right)}}{\Gamma(1+\min(\alpha_{1},\alpha_{2},\dots,\alpha_{n}))\left(\sum_{i=1}^{n}c_{i}t^{\sum_{j\neq i}^{n}\alpha_{i}-(n-1)\min(\alpha_{1},\alpha_2, \dots,\alpha_n)}\right)},\;\;   &{\rm as}\;t\rightarrow 0,\nonumber\\
		\frac{t}{\left(\sum_{i=0}^{n}c_{i}\alpha_i{\lambda_i}^{\alpha_i-1}\right)},\; \lambda_1,\lambda_2,\ldots,\lambda_{n} >0,\;\;   &{\rm as}\;t\rightarrow \infty,\nonumber\\	
	\frac{t^{\left(\sum_{i=1}^{n}{{{\alpha_{i}}}-(n-1)\min(\alpha_{1},\alpha_{2},\dots,\alpha_{n})}\right)}}{\Gamma(1+min(\alpha_{1},\alpha_{2},\dots,\alpha_{n}))\left(\sum_{i=1}^{n}{c_{i}t^{\sum_{j\neq i}^{n}{{{\alpha_{i}}}-(n-1)\min(\alpha_{1},\alpha_{2},\dots,\alpha_{n})}}}\right)},	\; \lambda_1,\lambda_2,\dots,\lambda_{n} =0, & {\rm as} \; t\rightarrow \infty.	
    \end{array}
\right.
\end{align}
Further, the corresponding $M_q(t)$ has the following asymptotic form
\begin{equation}
 	M_q(t){\sim}
\left\{
	\begin{array}{lll}
		\frac{t^{q\left(\sum_{i=1}^{n}{{{\alpha_{i}}}-(n-1)\min(\alpha_{1},\alpha_{2},\dots,\alpha_{n})}\right)}\Gamma(1+q)}{\Gamma(1+q\min(\alpha_{1},\alpha_{2},\dots,\alpha_{n}))\left(\sum_{i=1}^{n}{c_{i}t^{\sum_{j\neq i}^{n}{{{\alpha_{i}}}-(n-1)\min(\alpha_{1},\alpha_{2},\dots,\alpha_{n})}}}\right)^{q}}, \;& \mbox{as}\;t\rightarrow 0,\nonumber\\
		\frac{t^{q}}{\left(\sum_{i=0}^{n}{c_{i}\alpha_{i}{\lambda_{i}}^{\alpha_{i}-1}}\right)^{q}},\; \lambda_1,\lambda_2,\ldots,\lambda_{n} >0,\;\;   &\mbox{as}\;t\rightarrow \infty,\nonumber\\
		\frac{t^{q\left(\sum_{i=1}^{n}{{{\alpha_{i}}}-(n-1)\min(\alpha_{1},\alpha_{2},\dots,\alpha_{n})}\right)}\Gamma(1+q)}{\Gamma(1+q\min(\alpha_{1},\alpha_{2},\dots,\alpha_{n}))\left(\sum_{i=1}^{n}{c_{i}t^{\sum_{j\neq i}^{n}{{{\alpha_{i}}}-(n-1)\min(\alpha_{1},\alpha_{2},\dots,\alpha_{n})}}}\right)^{q}},	\; \lambda_1,\lambda_2,\dots,\lambda_{n} =0, & \mbox{as } t\rightarrow \infty. \\
    \end{array}
\right.
\end{equation}
 
\section{Time-changed Poisson process and Brownian motion}
In this section, we introduce time-changed Poisson process and Brownian motion by considering MTSS and IMTSS as time-changes. Note that Poisson process time-changed by MTSS generalize the space-fractional Poisson process (see Orsingher and Polito, 2012) and Poisson process time-changed by IMTSS generalize the time-fractional Poisson process (see e.g. Meerschaert et al. 2011) and references therein. Further, the Brownian motion time-changed by IMTSS generalize the Brownian motion time-changed by inverse stable subordinator model which is the scaling limit of continuous time random walk with infinite mean waiting time (see Meerschaert et al. 2009). It is worth to mention here that the governing equation of Brownian motion time-changed by inverse stable subordinator is a fractional analogous of heat equation which involves fractional derivative in time variable. We discuss the governing fractional differential equations of these time-changed processes.

\subsection{The mixture tempered-space fractional Poisson process (MTSFPP)}
In this section, we introduce and give the governing fractional difference-differential equation of mixture tempered space-fractional Poisson process (MTSFPP). A subordination representation of MTSFPP can be written as
\begin{equation}
X(t)= N(S_{\alpha_1,\lambda_1,\alpha_2,\lambda_2}(t)),
\end{equation}
where homogeneous Poisson process $N(t)$ with intensity $\mu >0$ is independent of $S_{\alpha_1,\lambda_1,\alpha_2,\lambda_2}(t)$.
The main purpose is to generalize a homogeneous Poisson process in fractional sense by introducing a fractional difference operator in the governing equation in the state space. The PMF $r(k,t) = P(X(t) =k)$ of MTSFPP can be easily obtained in an infinite series form by the standard conditioning argument. The probability generating function (PGF) $G(z,t)= \mathbb{E}[z^{ X(t)}]$ for $X(t)$ is given by
\begin{equation}\label{pgf-mtss}
G(z,t)= e^{-t (c_{1}\{(\lambda_{1}+\mu(1-z))^{\alpha_{1}}-{\lambda_{1}}^{\alpha_{1}}\} +c_{2}\{(\lambda_{2}+\mu(1-z))^{\alpha_{2}}-{\lambda_{2}}^{\alpha_{2}}\})}, \; |z|\leq 1, \;\; \mu \leq \frac{\lambda_{i}}{2},\; i=1,2.
\end{equation}

\begin{proposition}
The marginal distribution $r(k, t) = \mathbb{P}(X(t)=k)$ satisfies the following fractional difference differential equation
\begin{align}
\frac{d}{dt}r(k, t) = -\sum_{i=1}^{2} c_{i}\{(\lambda_{i} + \mu(1-B))^{\alpha_{i}}-\lambda_{i}^{\alpha_{i}}\} r(k,t),\; \alpha_{i} \in (0,1).
\end{align}
with the conditions $r(0,0)=1$ and $r(k, 0)=0$ for $k \neq 0$. 
\end{proposition}
\begin{proof}
Using the PGF, it follows 
\begin{align*}
\frac{\partial}{\partial t}G(z,t)&= -\sum_{i=1}^{2}c_{i}\left[\sum_{l=0}^{\infty}{\alpha_{i}\choose l}{\lambda_{i}}^{\alpha_{i}-l} 
{\mu}^{l}\sum_{m=0}^{\infty}{\l \choose m}
(-1)^{m}\sum_{k=0}^{\infty}{z}^{k}r_{k-m}(t) -{\lambda_{i}}^{\alpha_{i}}\sum_{k=0}^{\infty}z^{k}r_{k}(t)\right]\\
&= -\sum_{i=1}^{2}c_{i}\left[\sum_{l=0}^{\infty}{\alpha_{i}\choose l}{\lambda_{i}}^{\alpha_{i}-l}{\mu}^{l}\sum_{m=0}^{\infty}{\l \choose m}
(-z)^{m} -{\lambda_{i}}^{\alpha_{i}}\right]G(z,t)\\
&=-G(z,t)\sum_{i=1}^{2}c_{i}\left[(\lambda_{i}+\mu(1-z))^{\alpha_{i}}-{\lambda_{i}}^{\alpha_{i}}\right].
\end{align*}
The result follows by using $G(z,0)=1$ and \eqref{pgf-mtss}. 
\end{proof}

\subsection{The mixture tempered time-fractional Poisson process (MTTFPP)}
One can also define a mixture tempered time-fractional Poisson process (MTTFPP) by subordinating homogeneous Poisson process $N(t)$ with the IMTSS process $E_{\alpha_1,\lambda_1,\alpha_2, \lambda_2}(t)$ such as
\begin{equation} 
Y(t)= N(E_{\alpha_1,\lambda_1,\alpha_2, \lambda_2}(t)).
\end{equation}
Next, we derive the governing fractional difference differential equation for marginal distribution of MTTFPP $Y(t)$. 

\begin{proposition}
The marginal PMF $p_{\mu}(k,t)$ of $Y(t)$ satisfies the following governing equation
\begin{align*}
&\left[c_{1}\left(\lambda_1+\frac{\partial}{\partial t}\right)^{\alpha_1}+c_{2}\left(\lambda_2+\frac{\partial}{\partial t}\right)^{\alpha_2}\right]p_{\mu}(k,t)\\
 &= -\mu[p_{\mu}(k,t)-p_{\mu}(k-1,t)]+[\lambda_1^{\alpha_1}{c_1}+\lambda_2^{\alpha_1}{c_2}]p_{\mu}(k,t)\\
 &-c_1 t^{-\alpha_1}M_{1,1-\alpha_1}^{1-\alpha_1}(-\lambda_{1}t)\delta(x) -c_2 t^{-\alpha_2}M_{1,1-\alpha_2}^{1-\alpha_2}(-\lambda_{2}t)\delta(x),\;\; k\geq 1.
\end{align*}
\end{proposition}

\begin{proof}
Let $p(k,t)$ be the PMF of standard Poisson process. By standard conditioning argument and using \eqref{Gover_IMTSS}, we have
\begin{align*}
& \left[c_{1}\left(\lambda_1 +\frac{\partial}{\partial t}\right)^{\alpha_1}+ c_{2}\left(\lambda_2+\frac{\partial}{\partial t}\right)^{\alpha_2}\right]p_{\mu}(k,t) \\
& = \int_{0}^{\infty}p(k,u)\left[c_{1}\left(\lambda_1+\frac{\partial}{\partial t}\right)^{\alpha_1}+c_{2}\left(\lambda_2+\frac{\partial}{\partial t}\right)^{\alpha_2}\right]H(u,t)du\\
& = -\int_{0}^{\infty}p(k,u)\frac{\partial}{\partial u}H(u,t)du + [\lambda_{1} c_{1}+\lambda_{2} c_{2}]p_{\mu}(k, t)\\
&-c_1 t^{-\alpha_1}M_{1,1-\alpha_1}^{1-\alpha_1}(-\lambda_{1}t)\delta(x) -c_2 t^{-\alpha_2}M_{1,1-\alpha_2}^{1-\alpha_2}(-\lambda_{2}t)\delta(x),
\end{align*}
and finally integration by parts yields the desired result.
\end{proof}

\subsection{Time-changed Brownian motion}

In this secion, we introduce the time-changed processes Z(t) and W(t) as a Brownian motion $B(t)$ time-changed by MTSS $S_{\alpha_1,\lambda_1,\alpha_2,\lambda_2}(t)$ and IMTSS $E_{\alpha_1,\lambda_1,\alpha_2,\lambda_2}(t)$ respectively, i.e.
\begin{align}\label{Brow_Mtss}
Z(t) &= B(S_{\alpha_1,\lambda_1,\alpha_2,\lambda_2}(t)),\\
W(t) &= B(E_{\alpha_1,\lambda_1,\alpha_2,\lambda_2}(t)), \;\; t>0.
\end{align}
By applying the previous results, we can find the governing equations for the pdf of $Z(t)$ and $W(t)$ and the same can be generalized for $N$-th order mixtures of TSS.
\begin{proposition}
The pdf $r_{\alpha_1,\lambda_1,\alpha_2,\lambda_2}(x,t)= \mathbb{P}(Z(t) \in dx)$ of the time-changed Brownian motion $Z(t)$ defined in \eqref{Brow_Mtss}, satisfies the following space-fractional differential equation 
\begin{align}
\frac{\partial}{\partial t} r_{\alpha_1,\lambda_1,\alpha_2,\lambda_2}(x, t) & = -c_{1}\left(\lambda_1-\frac{\partial^{2}}{\partial x^{2}}\right)^{\alpha_1}r_{\alpha_1,\lambda_1,\alpha_2,\lambda_2}(x, t)-c_{2}\left(\lambda_2-\frac{\partial^{2}}{\partial x^{2}}\right)^{\alpha_2}r_{\alpha_1,\lambda_1,\alpha_2,\lambda_2}(x, t)\nonumber\\
&+\lambda_1^{\alpha_1}{c_1}r_{\alpha_1,\lambda_1,\alpha_2,\lambda_2}(x, t)+\lambda_2^{\alpha_1}{c_2}r_{\alpha_1,\lambda_1,\alpha_2,\lambda_2}(x, t),
\end{align}
with initial and boundary conditions
\begin{equation}
 \begin{cases}
   r_{\alpha_1,\lambda_1,\alpha_2,\lambda_2}(x, 0) = \delta(x)\\
   \lim_{|x|\rightarrow \infty}r_{\alpha_1,\lambda_1,\alpha_2,\lambda_2}(x, t)=0,\\
   \lim_{|x|\rightarrow \infty}\frac{\partial}{\partial x}r_{\alpha_1,\lambda_1,\alpha_2,\lambda_2}(x, t)=0.
 \end{cases}
\end{equation}
\end{proposition}
\begin{proof}
We will use proposition (3.4) to prove this result.
One can write 
$$
r_{\alpha_1,\lambda_1,\alpha_2,\lambda_2}(x, t) = \int_{0}^{\infty} q(x,u) G(u,t) du,
$$
where $q(x,t)$ is the pdf of the standard Brownian motion $B(t)$. Further,
\begin{align*}
\frac{\partial}{\partial t} r_{\alpha_1,\lambda_1,\alpha_2,\lambda_2}(x, t) & = \int_{0}^{\infty} q(x,u) \frac{\partial}{\partial t}G(u,t) du\\
&= (\lambda_1^{\alpha_1}{c_1} +\lambda_2^{\alpha_1}{c_2})r_{\alpha_1,\lambda_1,\alpha_2,\lambda_2}(x, t)- c_{1}\sum_{i=0}^{\infty}{\alpha_{1}  \choose i}\lambda_{1}^{\alpha_1-i}\int_{0}^{\infty} q(x,u)
 \frac{\partial^{i}}{\partial u^{i}}G(u,t) du\\
 &- c_{1}\sum_{j=0}^{\infty}{\alpha_{2}  \choose j}\lambda_{2}^{\alpha_2-j}\int_{0}^{\infty} q(x,u)
 \frac{\partial^{j}}{\partial u^{j}}G(u,t) du\\
& = (\lambda_1^{\alpha_1}{c_1} +\lambda_2^{\alpha_1}{c_2})r_{\alpha_1,\lambda_1,\alpha_2,\lambda_2}(x, t)- c_{1}\sum_{i=0}^{\infty}(-1)^{i}{\alpha_{1}  \choose i}\lambda_{1}^{\alpha_1-i}\int_{0}^{\infty}  \frac{\partial^{i}}{\partial u^{i}}q(x,u)
G(u,t) du\\
 &- c_{1}\sum_{j=0}^{\infty}{\alpha_{2}  \choose j}(-1)^{j}\lambda_{2}^{\alpha_2-j}\int_{0}^{\infty} \frac{\partial^{j}}{\partial u^{j}}q(x,u)
 G(u,t) du\\
 & =  (\lambda_1^{\alpha_1}{c_1} +\lambda_2^{\alpha_1}{c_2})r_{\alpha_1,\lambda_1,\alpha_2,\lambda_2}(x, t)\\
 &- \left[ c_{1}\sum_{i=0}^{\infty}{\alpha_{1}  \choose i}\lambda_{1}^{\alpha_1-i}\left(-\frac{\partial^{2}}{\partial x^{2}}\right)^{i}+c_{2}\sum_{j=0}^{\infty}{\alpha_{2}  \choose j}\lambda_{2}^{\alpha_1-j}\left(-\frac{\partial^{2}}{\partial x^{2}}\right)^{j}\right]r_{\alpha_1,\lambda_1,\alpha_2,\lambda_2}(x, t),
\end{align*}
hence proved.
\end{proof}

\begin{proposition}
The density $w_{\alpha_1,\lambda_1,\alpha_2,\lambda_2}(x, t)$ of the process $W(t)$ defined in \eqref{Brow_Mtss} satisfies the
following fractional differential equation
\begin{align*}\label{GOVE_BMTSS}
&\frac{\partial^{2}}{\partial x^{2}}w_{\alpha_1,\lambda_1,\alpha_2,\lambda_2}(x, t)\\
 & =-\left[c_{1}\left(\lambda_1+\frac{\partial}{\partial t}\right)^{\alpha_1}
+c_{2}\left(\lambda_2+\frac{\partial}{\partial t}\right)^{\alpha_2}\right]w_{\alpha_1,\lambda_1,\alpha_2,\lambda_2}(x, t) +\left[\lambda_1^{\alpha_1}{c_1}+\lambda_2^{\alpha_1}{c_2}\right]w_{\alpha_1,\lambda_1,\alpha_2,\lambda_2}(x, t)\nonumber \\
 &-c_1 t^{-\alpha_1}M_{1,1-\alpha_1}^{1-\alpha_1}(-\lambda_{1}t)\delta(x) -c_2 t^{-\alpha_2}M_{1,1-\alpha_2}^{1-\alpha_2}(-\lambda_{2}t)\delta(x),
\end{align*}
with initial and boundary conditions
\begin{equation*}
 \begin{cases}
    w_{\alpha_1,\lambda_1,\alpha_2,\lambda_2}(x, 0) = \delta(x)\\
    \lim_{|x|\rightarrow \infty}w_{\alpha_1,\lambda_1,\alpha_2,\lambda_2} (x, t)=0,\\
  \lim_{|x|\rightarrow \infty}\frac{\partial}{\partial x}w_{\alpha_1,\lambda_1,   \alpha_2,\lambda_2}(x, t)=0.
\end{cases}
\end{equation*}
\end{proposition}
\begin{proof}
Using similar argument as given in Proposition $7.3$ and with the help of \eqref{Gover_IMTSS}, the result can be proved.
\end{proof}

\section*{Ackowledgements} NG would like to thank Council of Scientific and Industrial Research(CSIR), India, for the award of a research fellowship.
\vone
\noindent
%\begin{namelist}{xxx}

%\end{namelist}
\end{document}